\documentclass[12pt]{amsart}
\usepackage{latexsym, url}
\usepackage{amsthm}
\usepackage{amsmath}
\usepackage{amsfonts}
\usepackage{amssymb}
\usepackage[dvips]{graphicx}
\usepackage{xypic}
\input xy
\xyoption{all}

\newtheorem{theo}{Theorem}[section]

\newtheorem{nota}[theo]{Notation}
\newtheorem{conv}[theo]{Convention}
\newtheorem{propo}[theo]{Proposition}
\newtheorem{defi}[theo]{Definition}
\newtheorem{coro}[theo]{Corollary}
\newtheorem{rem}[theo]{Remark}

\newtheorem{exams}[theo]{Examples}

\newcommand\Inj{\operatorname{Inj}}

\newcommand\Alg{\operatorname{Alg}}
\newcommand\Coalg{\operatorname{Coalg}}

\newcommand\id{\operatorname{id}}

\newcommand\Pos{\operatorname{\bf Pos}}

\newcommand\Bool{\operatorname{\bf Bool}}

\newcommand\Comb{\operatorname{Comb}}

\newcommand\cof{\operatorname{cof}}

\newcommand\cell{\operatorname{cell}}

\newcommand\cc{\mathcal {C}}

\newcommand\cf{\mathcal {F}}

\newcommand\ci{\mathcal {I}}

\newcommand\cj{\mathcal {J}}
\newcommand\ck{\mathcal {K}}
\newcommand\cl{\mathcal {L}}
\newcommand\cm{\mathcal {M}}
\newcommand\crr{\mathcal {R}}
\newcommand\cs{\mathcal {S}}

\newcommand\cp{\mathcal {P}}
\newcommand\cw{\mathcal {W}}
\newcommand\cx{\mathcal {X}}

 \newbox\noforkbox \newdimen\forklinewidth
\forklinewidth=0.3pt \setbox0\hbox{$\textstyle\smile$}
\setbox1\hbox to \wd0{\hfil\vrule width \forklinewidth depth-2pt
 height 10pt \hfil}
\wd1=0 cm \setbox\noforkbox\hbox{\lower 2pt\box1\lower
2pt\box0\relax}

\date{May 27, 2020}
 
\begin{document}
\title[Remarks on combinatorial and accessible model categories]
{Remarks on combinatorial and accessible model categories}
\author[J. Rosick\'{y}]
{J. Rosick\'{y}}
\thanks{Supported by the Grant Agency of the Czech Republic under the grant 
              19-00902S.} 
\address{
\newline J. Rosick\'{y}\newline
Department of Mathematics and Statistics\newline
Masaryk University, Faculty of Sciences\newline
Kotl\'{a}\v{r}sk\'{a} 2, 611 37 Brno, Czech Republic\newline
rosicky@math.muni.cz
}
 
\begin{abstract}
Using full images of accessible functors, we prove some results about combinatorial and accessible model categories. In particular,
we give an example of a weak factorization system on a locally presentable category which is not accessible. 
\end{abstract} 

\maketitle

\section{Introduction}
Twenty years ago, M. Hovey asked for examples of model categories which are not cofibrantly generated. This is the same as asking
for examples of weak factorization systems which are not cofibrantly generated. One of the first examples was given in \cite{AHRT}:
it is a weak factorization system $(\cl,\crr)$ on the locally presentable category of posets where $\cl$ consists of embeddings.
The reason is that posets injective to embeddings are precisely complete lattices which do not form an accessible category. Hence
$\cl$ is not cofibrantly generated. Since then, an importance of accessible model categories and accessible weak factorization systems has emerged, And, the same question appears again, i.e., to give an example of a weak factorization system $(\cl,\crr)$ on a locally presentable category which is not accessible. Now, $\cl$-injective objects do not necessarily form an accessible category but only a full image
of an accessible functor. Such full images are accessible only under quite restrictive assumptions (see \cite{BTR}). But, for an accessible weak factorization system, $\cl$-injective objects form the full image of a forgetful functor from algebraically $\cl$-injective objects. Such full images are closed under reduced products modulo $\kappa$-complete filters for some regular cardinal
$\kappa$. We use this property to present a non-accessible factorization system on the category of Boolean algebras having, again,
$\cl$ consisting of embeddings. Full images of accessible functors are also used for showing that accessible weak factorization systems
on a locally presentable category are closed under small intersections. Another proof of this fact is given in \cite{H1}. 

Given a cofibrantly weak factorization system $(\cl,\crr)$ on a locally presentable category $\ck$, \cite{RT2} constructs a class 
$\cw_{\cl}$ and shows that, assuming Vop\v enka's principle, $\cw_{\cl}$ is the smallest class of weak equivalences making $\ck$
a model category with $\cl$ as the class of cofibrations. There is still open whether Vop\v enka's principle is needed for this.  Recently, S. Henry \cite{H} has radically generalized results of C.-D. Cisinski \cite{C} and M. Olschok \cite{O} and has given mild assumptions under which Vop\v enka's principle is not needed. Using full images of accessible functors we show that $(\ck,\cl,\cw_\cl)$  is a model category iff its transfinite construction from \cite{RT2} converges, i.e., it stops at some ordinal.

Finally, we show that weak equivalences in an accessible model category form a full image of an accessible functor, which corrects
an error in \cite{R3}.

{\bf Acknowledgement.} We are grateful to J. Bourke for valuable discussions about this paper.

\section{Full images} 
Let $F:\cm\to\ck$ be an accessible functor. Recall that this means that both $\cm$ and $\ck$ are accessible and $F$
preserves $\lambda$-directed colimits for some regular cardinal $\lambda$. The full subcategory of $\ck$ consisting 
of objects $FM$, $M\in\cm$ is called a \textit{full image} of $F$. While accessible categories are, up to equivalence,
precisely categories of models of basic theories, full images of accessible functors are, up to equivalence, precisely
categories of structures which can be axiomatized using additional operation and relation symbols (see \cite{R1});
they are also called pseudoaxiomatizable. In both cases, we use infinitary first-order theories.

Let $\cm$ be a full subcategory of a category $\ck$ and $K$ an object in $\ck$. We say that $\cm$ satisfies 
the \textit{solution-set condition} at $K$ if there exists a set of morphisms $(K\to M_i)_{i\in I}$ with $M_i$ 
in $\cm$ for each $i\in I$ such that every morphism $f:K\to M$ with $M$ in $\cm$ factorizes through some $f_i$, 
i.e., $f=gf_i$. $\cm$ is called \textit{cone-reflective} in $\ck$ if it satisfies the solution-set condition
at each object $K$ in $\ck$ (see \cite{AR}). Given a set $\cx$ of objects of $\ck$, we say that $\cm$ 
satisfies the solution set condition at $\cx$ if it satisfies this condition at each $X\in\cx$.  

\begin{propo}[\cite{R} 2.4]\label{cone}
The full image of an accessible functor is cone-reflective in $\ck$.
\end{propo}

\begin{propo}\label{union}
Let $\ck$ be a locally presentable category, $I$ a set and $\cx_i\subseteq\ck$, $i\in I$, full images of accessible functors.
Then $\cup_{i\in I}\cx_i$ is a full image of an accessible functor.
\end{propo}
\begin{proof}
Let $\cx_i$ be full images of accessible functors $F_i:\cm_i\to\ck$, $i\in I$. Then $\cup_{i\in I} \cx_i$ is a full image of an accessible functor $F:\coprod_{i\in I} \cx_i\to\ck$ induced by $F_i$. 
\end{proof}

\begin{nota}
{
\em
Let $\cx$ be a class of morphisms in $\ck$. Then $\overline{\cx}$ will denote its 2-out-of-3 closure, i.e., the smallest class of morphisms
such that
\begin{enumerate}
\item $f,g\in\overline{\cx}$ implies $gf\in\overline{\cx}$,
\item $gf,f\in\overline{\cx}$ implies $g\in\overline{\cx}$ and
\item $gf,g\in\overline{\cx}$ implies $f\in\overline{\cx}$.
\end{enumerate}
We will consider these classes as full subcategories in $\ck^{\to}$.
}
\end{nota}

\begin{propo}\label{closure}
Let $\ck$ be a locally presentable category and $\cx\subseteq\ck^{\to}$ a full image of an accessible functor. Then $\overline{\cx}$ is a full image of an accessible funstor.
\end{propo}
\begin{proof}
$\overline{\cx}$ can be obtained from $\cx$ by a sequence of pseudopullbacks. Let $\cx_0=\cx$. We take composable pairs 
of $\cx_0$ and their compositions form $\cx_1$. Then we take those pairs $(g,f)$ for which $f$ and the composition $gf$ belong to $\cx_1$. Their 
$g$'s form $\cx_2$. Further we take those pairs $(g,f)$ for which $g$ and $gf$ belong to $\cx_2$. Their $f$'s form $\cx_3$. By iterating this construction, we get $\overline{\cx}=\cup_{i<\omega} \cx_i$, Thus the result follows from \cite{R} 2.6 and \ref{union}.
\end{proof}

\begin{propo}\label{reduced}
Let $F:\cm\to\ck$ be a limit preserving $\kappa$-accessible functor where $\cm$ is locally $\kappa$-presentable. Then the full image
of $F$ is closed in $\ck$ under reduced products modulo $\kappa$-complete filters.
\end{propo}
\begin{proof}
Let $I$ be a set, $K_i=F(M_i)$, $i\in I$ and let $\cf$ be a $\kappa$-complete filter on $I$. Then the reduced product 
$\prod_\cf K_i$ is a $\kappa$-directed colimit of projections $K_i^A\to K_i^B$ where $A\supseteq B\in\cf$. Then $K=F(M)$
where $M=\prod_\cf K_i$. 
\end{proof}
 
\section{Accessible weak factorization systems} 
A functorial weak factorization system in a locally presentable category is called \textit{accessible} if the factorization functor
$F:\ck^\to\to\ck^{\to\to}$ is accessible (see \cite{R3}. Here, $\ck^{\to\to}$ denotes the category of composable pairs of morphisms.
Any cofibrantly generated weak factorization system in a locally presentable category is accessible.  
 
\begin{propo}\label{acc} 
Let $(\cl,\crr)$ be an accessible weak factorization system in a locally presentable category $\ck$. Then $\crr$ is a full image 
of a limit-preserving accessible functor $\cm\to\ck^{\to}$ where $\cm$ is locally presentable.
\end{propo}
\begin{proof}
$\crr$ is the full image of an accessible functor $\Alg(R)\to\ck^\to$ (see \cite{R3} 2.3(2) and 4.2(1)).
\end{proof} 

\begin{rem}\label{acc1}
{
\em
But $\crr$ does not need to be accessible, see \cite{R3} 2.6. Thus \cite{R3} 5.2 (1) is not correct (I am indepted to M. Shulman for pointing this up). Neither it is accessibly embedded to $\ck^2$. Assuming the existence of a proper class of almost strongly compact cardinals, $\crr$ is preaccessible and \textit{preaccessibly embedded} to $\ck^\to$; see the proof of \cite{R3} 2.2. The latter means that the embedding $\crr\to\ck^2$ preserves $\lambda$-directed colimits for some $\lambda$.   
}
\end{rem}

\begin{coro}\label{acc2}
Let $(\cl,\crr)$ be an accessible weak factorization system in a locally presentable category $\ck$. Then $\cl$-$\Inj$ is a full image
of a limit-preserving accessible functor $\cm\to\ck$ where $\cm$ is locally presentable.
\end{coro}
\begin{proof}
An object $K$ is $\cl$-injective if and only if $K\to 1$ is in $\crr$. $R:\ck^2\to\ck^2$ restricts to a limit-preserving accessible functor on $\ck\downarrow 1$ and $\cl$-$\Inj$ is the full image of this restriction.
\end{proof}

The next result improves Proposition 3.4 on \cite{R}.

\begin{propo}\label{cofgen}
Let $(\cl,\crr)$ be an accessible weak factorization system in a locally presentable category $\ck$. Then $\cl$ is a full image of an colimit-preserving (accessible) functor $\cm\to\ck^\to$ where $\cm$ is locally presentable.
\end{propo}
\begin{proof}
$\crr$ is the full image of an accessible functor $\Coalg(L)\to\ck^\to$ (see \cite{R3} 2.3(2) and 4.2(1)).
\end{proof}

\begin{rem}
{
\em
If $(\cl,\crr)$ is a cofibrantly generated weak factorization system in a locally presentable category then $\cl$ does not need to
be accessible. An example is given in \cite{R} 3.5(2) under the axiom of constructibility. In this example, $\cl$ is accessible assuming
the existence of an almost strongly compact cardinal. We do not know any example of non-accessible $\cl$ in ZFC. Example \cite{R}
3.3(1) is not correct because split monomorphisms are not cofibrantly generated in posets (this was pointed up by T. Campion).
}
\end{rem}

\begin{rem}
{
\em
A weak factorization system $(\cl,\crr)$ is accessible iff $\Coalg(L)$ is locally presentable, which is a kind of smallness property
of  $\Coalg(L)$.  On the other hand, cofibrant generation is a smallness property of $\cl$. It does not seem that accessibility
is a smallness property of $\cl$. For instance, $\cl$ in the next Example is finitely accessible.
}
\end{rem}

\begin{exams}
{
\em
(1) Let $\cl$ be the class of regular monomorphisms (= embeddings) in the category $\Bool$ of Boolean algebras. Then $\cl$-injective
Boolean algebras are precisely complete Boolean algebras and $\Bool$ has enough $\cl$-injectives (see \cite{Ha}). Thus 
$(\cl,\cl^\square)$ is a weak factorization system (see \cite{AHRT}, 1.6). We will show that this weak factorization system is not accessible. Following \ref{reduced} and \ref{acc2}, it suffices to show that complete Boolean algebras are not closed under reduced products modulo $\kappa$-complete filters for any regular cardinal $\kappa$. I have learnt the following example from M. Goldstern.

Let $I$ be a set of cardinality $\kappa$ and $\cf$ be the filter of subsets $X\subseteq I$ such that the cardinality of $I\setminus X$
is $<\kappa$. Then the reduced product $\prod_\cf 2$ is isomorphic to the Boolean algebra $U(\kappa)=\cp(I)/[\kappa]^{<\kappa}$ where
$[\kappa]^{<\kappa}$ is the ideal $J$ consisting of subsets of cardinality $<\kappa$. Let $A_i$, $i<\kappa$ be pairwise disjoint subsets of $I$ of cardinality $\kappa$. Let $X$ be an upper bound of $A_i$, $i<\kappa$ in $U(\kappa)$. Choose $a_i\in A_i$, $i<\kappa$.
Then $X\setminus\{a_i|\, i<\kappa\}$ is an upper bound of $A_i$ in $U(\kappa)$ smaller than $X$. Hence $A_i$, $i<\kappa$ do not have a supremum in $U(\kappa)$.

(2) Let $\cl$ be the class of regular monomorphisms (= embeddings) in the category $\Pos$ of posets. Then $\cl$-injective
posets are precisely complete lattices and $\Pos$ has enough $\cl$-injectives (see \cite{BB}). Since the forgetful functor
$\Bool\to\Pos$ preserves products and directed colimits, complete lattices are not closed under reduced products modulo $\kappa$-complete filters for any regular cardinal $\kappa$. It suffices to take the same reduced products as in (1).
}
\end{exams}

\begin{rem}
{
\em
We can order weak factorization systems: $(\cl_1,\crr_1)\leq (\cl_2,\crr_2)$ if $\cl_1\subseteq\cl_2$. Following \cite{R3}, 4.3
accessible weak factorization systems have small joins: if $\cl_i$ is generated by $\cc_i$, $i\in I$ then $\cup_{i\in I}\cl_i$
is generated by $\cup_{i\in I}\cc_i$. S. Henry \cite{H1} showed that they have small meets. We will give another proof.
}
\end{rem}

\begin{propo}
Let $(\cl_i,\crr_i)$, $i\in I$ be a set of accessible weak factorization systems in a locally presentable category. Then 
$(\cap_{i\in I} \cl_i,\crr)$ is an accessible weak factorization system.
\end{propo}
\begin{proof}
Let $\cp$ be the pseudopullback of all forgetful functors $\Coalg(L_i)\to\ck^\to$ (see \cite{R3}). Then $\cp$ is locally presentable and the full image of $U:\cp\to\ck^\to$ is $\cl=\cap_{i\in I}\cl_i$ (see \cite{R}, 2.6). There is a regular cardinal $\lambda$ such that $\cp$ is locally $\lambda$-presentable and $U$ preserves $\lambda$-filtered colimits. Let $\cc$ be the (representative) full subcategory of $\lambda$-presentable objects in $\cp$. Following \cite{R3}, 3.3, $\cc^\boxplus = \cp^\boxplus$ and thus, for
$\crr=|\cp^\boxplus|$, $({}^\square\crr,\crr)$ is an accessible weak factorization system (see \cite{R3}, 3.6 and 4.3).
It remains to show that ${}^\square\crr=\cl$.

Since
$$
\crr_i = |\Coalg(L_i)^\boxplus|\subseteq |\cp^\boxplus|=\crr,
$$
we have $\cup_{i\in I} \crr_i\subseteq \crr$. Hence
$$
{}^\square\crr\subseteq {}^\square(\cup_{i\in I} \crr_i)\subseteq \cap_{i\in I} {}^\square \crr_i\subseteq \cl.
$$
On the other hand,
$$
\cl = |\cp|\subseteq |{}^\boxplus(\cp^\boxplus)|\subseteq {}^\square|\cp^\boxplus|\subseteq {}^\square\crr.
$$

\end{proof}

\section{Combinatorial model categories}
\begin{conv}\label{convention}
{
\em
In what follows, $(\cl,\crr)$ will be a weak factorization system in a locally presentable category $\ck$ cofibrantly generated by $\ci$.
}
\end{conv}
Denote by $\Comb(\cl)$ the class of all combinatorial model structures with $\cl$ as cofibrations. We can order it by $(\cl,\cw_1)\leq (\cl,\cw_2)$ iff $\cw_1\subseteq\cw_2$.

\begin{propo}[\cite{R} 4.7]\label{meet}
$\Comb(\cl)$ has small meets given as 
$$
\wedge_{i\in I}(\cl,\cw_i)=(\cl,\cap_{i\in I}\cw_i).
$$.
\end{propo}

\begin{rem}\label{join}
{
\em
(1) Consider $(\cl,\cw_i)\in\Comb(\cl)$ where $I\neq\emptyset$ such that $(\cl,\cw_{i_0})$ is left proper for some $i_0\in I$. Each 
$\cw_i\cap\cl$ is cofibrantly generated by a set $\cs_i$. Put $\cs=\cup_{i\in I}(\cs_i)\setminus\cs_{i_0}$. Then the left Bousfield localization of $(\cl,\cw_{i_0})$ at $\cs$, yields the join $\vee_{i\in I}(\cl,\cw_i)$ in $\Comb(\cl)$.

(2) Assuming Vop\v enka's principle, $\Comb(\cl)$ is a large complete lattice, i.e., it has all joins and meets. In particular, $\Comb(\cl)$ has the smallest element. There are given as
$\vee_{i\in I}(\cl,\cw_i)=(\cl,\cup_{i\in I}\cw_i)$ and $(\wedge_{i\in I}(\cl,\cw_i)=(\cl,\cap_{i\in I}\cw_i)$. This follows from Smith's theorem because, assuming Vop\v enka's principle, every full subcategory of a locally presentable category has a small dense subcategory.
Thus it is cone-reflective. Add that Vop\v enka's principle is equivalent to the statement that any full subcategory of a locally
presentable category is cone-reflective (see \cite{AR} 6.i, or \cite{RT1}, 1.2(2)).
}
\end{rem}

\begin{defi}[\cite{RT2} 2.1]\label{left}
{
\em
Let $\cw_\cl$ be the smallest class $\cw$ of morphisms such that
\begin{enumerate}
\item $\crr\subseteq\cw$,
\item $\cw$ satisfies the 2-out-of-3 condition, and
\item $\cl\cap\cw$ is closed under pushout, transfinite composition and retracts.
\end{enumerate}
If $(\cl,\cw_\cl)$ is a model structure it is called \textit{left-determined}. 
}
\end{defi}
 
\begin{rem}\label{left1}
{
\em
(1) Retracts are meant in the category of morphisms $\ck^\to$. \cite{RT2} assumes in (2) that $\cw$ is closed under retracts. But this can be omitted following \cite{Ri} (or Lemma 1 in Model category, nLab). On the other hand, we assume it in (3).

In what follows $\cof(\cx)$ will denote the closure of $\cx$ under pushout, transfinite composition and retracts while $\cell(\cx)$ the closure under pushout and transfinite composition.

(2) If $(\cl,\cw_\cl)$ is a combinatorial model category, it is the smallest element in $\Comb(\cl)$. It always happens assuming 
Vop\v enka's principle. But, without it, we do not know whether the smallest element in $\Comb(\cl)$ might exist without being equal 
to $\cw_\cl$.

Recently, S. Henry \cite{H} proved the existence of a left-determined model structure in ZFC under mild assumption. 
}
\end{rem}

\begin{nota}\label{iterate}
{
\em
We put $\cw_0=\crr$, $\cw_{i+1}=\overline{\cw_i}$ if $i$ is an
even ordinal, $\cw_{i+1}=\cof(\cl\cap\cw_i)\cup\cw_i$ if $i$ is an odd ordinal and $\cw_i=\cup_{j<i}\cw_j$ if $0<i$ is a limit ordinal. Recall that any limit ordinal is even and $i+1$ is odd iff $i$ is even. Then $\cw_\cl=\cup_i \cw_i$ where $i$ runs over all ordinals.

We say that $\cw_i$ \textit{stops} if $\cw_\cl=\cw_i$ for an ordinal $i$.
}
\end{nota}

\begin{theo}\label{stop}
$(\cl,\cw_\cl)$ is a combinatorial model category iff $\cw_i$ stops.
\end{theo}
\begin{proof}
I. Assume that $(\cl,\cw_\cl)$ is a combinatorial model structure. Then $\cl\cap\cw_\cl$ is cofibrantly generated by a set $\cs$.
There is an odd ordinal $i$ such that $\cs\subseteq\cw_i$. Thus $\cl\cap\cw_\cl\subseteq\cw_{i+1}$. Hence $\cw_\cl\subseteq\cw_{i+2}$
and the construction stops.

II. Assume that $\cw_i$ stops. At first, we replace $\cw_i$ by $\cw^\ast_i$ which are full images of accessible functors. They are  defined in the same way as $\cw_i$ for $i$ even. $\crr$ is a full image of an accessible functor (following \cite{R} 3.3) and 2-out-of-3 closure and union keep full images of accessible functors (see \ref{closure} and \ref{union}). Let $i$ be odd. We will follow the proof of Smith's theorem given in \cite{B}. Since $\cw_i$ is cone-reflective (see \ref{cone}) and satisfies the 2-out-of-3 property, \cite{B} 1.9 produces a set $\cj$ needed for \cite{B} 1.8 for $\cl$ and $\cw_i$. We put $\cw^\ast_{i+1}=\cof(\cj)\cup\cw^\ast_i$. Then
$\cw^\ast_{i+1}\subseteq \cw_{i+1}$. Following \ref{cofgen} and \ref{union}, $\cw^\ast_{i+1}$ is a full image of an accessible functor. Like in Corollary of this lemma, we take $f\in\cw_i$ and express it as $f=hg$ with $g\in\cell(\cj)$ and $h\in\crr$. Thus there exist $t$ such that $tf=g$ and $ht=\id$. Hence $t\in\cw^\ast_1$ and $g\in\cw^\ast_{i+1}$. Thus $f\in\cw^\ast_{i+2}$. Therefore 
$\cw_{i}\subseteq\cw^\ast_{i}\subseteq\cw_{i+2}$. Consequently, $\cw=\cup_{i}\cw^\ast_i$. 

Since $\cw_i$ stops, $\cw^\ast_i$ stops as well. Hence $\cw$ is a full image of an accessible functor and thus it is cone-reflective 
(see \ref{cone}). Smith's theorem implies that $(\cl,\cw)$ is a combinatorial model category.
\end{proof}

\section{Accessible model categories}
 
A model category $(\cc,\cw)$ on a locally presentable category $\ck$ is \textit{accessible} if both $(\cc,\cc^\square)$ and 
$(\cc\cap\cw,(\cc\cap\cw)^\square)$ are accessible weak factorization systems.

\begin{propo}\label{modacc}
Let $(\cc,\cw)$ be an accessible model category on a locally presentable category $\ck$. Then $\cw$ is a full image of an accessible functor. 

Assuming the existence of a proper class of almost strongly compact cardinals, $\cw$ is preaccessible and preaccessibly embedded 
to $\ck^\to$.
\end{propo}
\begin{proof}
The first claim is what \cite{R3} 5.2(2) proves, using \cite{R3} 2.6. The second claim follows from \cite{R3} 2.2.  
\end{proof}

\begin{rem}\label{modacc2}
{
\em
To correct \cite{R3} 5.3, one has to replace (4) by

(4') $\cw$ is preaccessible and preaccessibly embedded to $\ck^\to$.

Indeed, in the proof, $\cp$ is preaccessible and preaccessibly embedded to $\ck^\to$ and thus it has a small dense subcategory $\cj$
of $\lambda$-presentable objects. This is what the proof needs. Add that we can apply \cite{R3}, 3.3  because the forgetful functor
$\cp\to\ck^\to$ preserves $\lambda$-directed colimits.
}
\end{rem}

\end{document}